\documentclass[12pt,reqno]{amsart}
\usepackage[margin=2.5cm]{geometry}
\usepackage{amsfonts,amssymb,amsthm, mathtools}
\usepackage{amsmath}
\usepackage{datetime}
\usepackage{tabularx}
\everymath{\displaystyle}
\usepackage{xcolor}
\usepackage{enumerate,physics}
\usepackage{multirow}
\usepackage{float}
\everymath{\displaystyle}

\definecolor{fgreen}{RGB}{44,144, 14}

\numberwithin{equation}{section} 
\newtheorem{theorem}{Theorem}[section] 
\newtheorem{proposition}[theorem]{Proposition} 
 
\newtheorem{lemma}[theorem]{Lemma} 
\theoremstyle{definition}
\newtheorem{definition}[theorem]{Definition} 
\newtheorem{remark}[theorem]{Remark} 
\newtheorem{example}[theorem]{Example}

\usepackage[colorlinks=true,citecolor=cyan,  urlcolor=blue,linkcolor=blue]{hyperref}

\def\C{\mathbb C}

\def\C{\mathbb {C}}

\newcommand{\thmref}[1]{Theorem~\ref{#1}}

\begin{document}

	\title[On the Product of coninvolutory Affine Transformations]{On the Product of coninvolutory Affine Transformations}
	\author[S.  Dutta, K. Gongopadhyay \and R. Mondal]{Sandipan Dutta, Krishnendu Gongopadhyay \and 
		Rahul Mondal}
	
	\address{Mizoram University,
		Tanhril, Aizawl 796004, Mizoram, India}
	\email{sandipandutta98@gmail.com, mzut330@mzu.edu.in}
	
	\address{Indian Institute of Science Education and Research (IISER) Mohali,
		Knowledge City,  Sector 81, S.A.S. Nagar 140306, Punjab, India}
	\email{krishnendu@iisermohali.ac.in}

	\address{Indian Institute of Science Education and Research (IISER) Mohali,
		Knowledge City,  Sector 81, S.A.S. Nagar 140306, Punjab, India}
	\email{canvas.rahul@gmail.com}

	
	\date{ @\currenttime , \today}
   \begin{abstract}
A complex matrix is called \emph{coninvolutory} if 
$T\overline{T}=I$. In this paper, we study decompositions of affine 
transformations in $\mathrm{Aff}(n,\mathbb{C})=\mathrm{GL}(n,\mathbb{C})\ltimes \mathbb{C}^n$ 
into products of coninvolutions. 
We prove that an affine transformation \(g\) is a product of two coninvolutions in \(\mathrm{Aff}(n,\C)\) if and only if its linear part \(L(g)\) is \(c\)-reversible; that is, \(L(g)\) is conjugate to \(\overline{L(g)}^{-1}\) in \(\mathrm{GL}(n,\C)\). Equivalently, \(g\) is conjugate to \(\overline{g}^{-1}\) in \(\mathrm{Aff}(n,\C)\). We further 
characterize elements that are products of three coninvolutions via 
consimilarity and show that every $g=(A,v)\in \mathrm{Aff}(n,\mathbb{C})$ with 
$|\det(A)|=1$ can be expressed as a product of at most four coninvolutions.
\end{abstract}
   
    	\maketitle 

	\section{Introduction} 

The decomposition of individual elements into products of involutions (elements of order two) is a classical problem in geometry  and linear algebra. 
More generally, the study of products of two involutions \cite{dj,ghr} is closely connected to the concept of 
\emph{reversibility} in groups: an element \(A\) is said to be \emph{reversible} in a group \(G\) if it is conjugate to its inverse, and \emph{strongly reversible} if the reversing element can itself be chosen to be an involution in \(G\). 
Equivalently, an element is strongly reversible if and only if it can be expressed as a product of two involutions in \(G\); see \cite{fs,wo} for a detailed exposition of reversibility in diverse mathematical contexts.

For the group of invertible matrices over the complex numbers, the decomposition of elements into products of involutions has attracted considerable attention from diverse perspectives. 
Liu \cite{li} established necessary conditions for a matrix to be representable as a product of three involutions, 
while Kn\"uppel and Nielsen \cite{kn} showed that any element in \(\mathrm{SL}(V)\) can be expressed as a product of 
at most four involutions.

Note that the field of complex numbers carries a natural involution given by complex 
conjugation. A matrix \(T \in \mathrm{M}(n,\mathbb{C})\) is called \emph{coninvolutory} if 
\(T\overline{T}=I\), see \cite{gr}. This naturally raises the question of whether matrices over 
\(\mathbb{C}\) can be decomposed into products of coninvolutions. 
Although the decomposition of matrices into products of classical involutions is well-studied, 
the analogous problem for coninvolutions has received comparatively less attention.{ In \cite{hm,bhh} the authors have discussed this problem regarding coninvolutions.}

Ballantine \cite{ba} proved that a matrix \(T\) is a product of four 
coninvolutions if and only if \(|\det(T)|=1\).   Abara, Merino, and Paras 
\cite{amp} showed that \(T\) is a product of two coninvolutions if and only if 
\(T\) is similar to \({\overline{T}}^{-1}\). 
Coninvolutions are also closely related to the notions of consimilarity and 
condiagonalizability; see \cite{hh88,hh86,ik}.

More recently, in the context of the classification of projective transformations, \cite{gm} introduced the notion of conjugate-reversible (or $c$-reversible) elements in $\mathrm{SL}(n,\mathbb{C})$. 
An element $A \in \mathrm{SL}(n,\mathbb{C})$ is called \emph{$c$-reversible} if 
$A$ is conjugate to $\overline{A}^{-1}$, and \emph{strongly $c$-reversible} if the 
reversing element can be chosen to be coninvolutory in $\mathrm{SL}(n,\mathbb{C})$. They proved that every 
$c$-reversible element in $\mathrm{SL}(n,\mathbb{C})$ is strongly $c$-reversible 
and characterized such elements using trace invariants. Consequently, the 
classification of products of two coninvolutions reduces to the study of 
strongly $c$-reversible elements.

A similar notion of \(c\)-reversibility can be introduced for affine 
transformations, providing a natural framework for studying the decomposition 
of affine maps into products of coninvolutions. 

Let \(\mathrm{Aff}(n, \mathbb{C})\) denote the group of affine transformations 
of \(\mathbb{C}^n\), given by
\[
g: x \mapsto Ax + v, \quad A \in \mathrm{GL}(n, \mathbb{C}), \; v \in \mathbb{C}^n.
\]
This group can be identified with the semidirect product
\[
\mathrm{Aff}(n, \mathbb{C}) \cong \mathrm{GL}(n, \mathbb{C}) \ltimes \mathbb{C}^n.
\]
We denote such an affine transformation $g$  by $(A, v)$, and  linear part $A$ of $g$ is denoted by $L(g)$. 
 \begin{definition}\label{def:crev}
{An element \(g=(A,v)\in \mathrm{Aff}(n,\mathbb{C})\) is called \(c\)-reversible if there exists \(h\in \mathrm{Aff}(n,\mathbb{C})\) such that \(hgh^{-1}=\overline{g}^{-1}\), where \(\overline{g}=(\overline{A},\overline{v})\); it is called strongly \(c\)-reversible if, in addition, \(h\overline{h}=e=(I_n,0)\), in which case \(h\) is called a coninvolution.}
\end{definition}
In this paper, we investigate \(c\)-reversible and strongly \(c\)-reversible 
elements in the affine group \(\mathrm{Aff}(n,\mathbb{C})\). This work is partly motivated by a problem posed by O'Farrell and Short 
\cite[p.~78--79]{fs}. Reversibility in the Euclidean isometry group 
\(\mathrm{O}(n,\mathbb{R}) \ltimes \mathbb{R}^n\) was studied in \cite{sh}, 
and was later extended in \cite{gl} to the complex and quaternionic settings. 
As a result, reversibility is now well-understood in the groups 
\(\mathrm{U}(n,\mathbb{F}) \ltimes \mathbb{F}^n\), where 
\(\mathbb{F} = \mathbb{R}, \mathbb{C}, \mathbb{H}\).

The reversible and strongly reversible elements of the affine group were 
subsequently classified in \cite{glm} for the  affine group 
\(\mathrm{Aff}(n,\mathbb{F})\).
In particular, it was shown in \cite{glm} that every element of 
\(\mathrm{Aff}(n,\mathbb{F})\) can be expressed as a product of at most four 
involutions.

Building on the study of \(c\)-reversibility in the linear group $\mathrm{GL}(n, \C)$, 
in the present paper we investigate \(c\)-reversible and strongly 
\(c\)-reversible elements in the affine group \(\mathrm{Aff}(n,\mathbb{C})\). 
Our goal is to understand the decomposition of affine transformations 
into products of coninvolutions and to determine the minimal number of 
coninvolutions required to express an arbitrary element of the group.

The main results of this paper are summarized in the following three theorems. First, the following theorem provides a complete classification of strongly \(c\)-reversible 
elements. 

\begin{theorem}\label{Thm1}
Let \( g \in \mathrm{Aff}(n,\mathbb{C}) \). Then the following are equivalent:
\begin{enumerate}[(i)]
    \item \( g \) is \(c\)-reversible.
    
    \item \(g\) is strongly $c$-reversible, i.e. \( g \) is a product of two coninvolutions.
    
    \item \( L(g) \) is \(c\)-reversible in \( \mathrm{GL}(n,\mathbb{C}) \).

\end{enumerate}
\end{theorem}

Consequently, the problem of decomposing an affine transformation into a 
product of two coninvolutions reduces to the study of the \(c\)-reversibility 
of its linear part. Next we determine when an affine transformation is a product of three coninvolutions {and use consimilarity of two matrices for this purpose}. 

\begin{definition}\label{def:consimilar}
An element \(g \in \mathrm{Aff}(n,\mathbb{C})\) is said to be 
\emph{consimilar} to \(h \in \mathrm{Aff}(n,\mathbb{C})\) if there exists 
\(k \in \mathrm{Aff}(n,\mathbb{C})\) such that
\[
h = k g \overline{k}^{-1}.
\]
\end{definition}

Note that consimilarity defines an equivalence relation on 
\(\mathrm{Aff}(n,\mathbb{C})\).

\begin{theorem}\label{Thm2}
Let \(g = (A,v) \in \mathrm{Aff}(n,\mathbb{C})\) satisfy \(|\det(A)| = 1\). 
Then \(g\) can be expressed as a product of three coninvolutions if and only 
if \(g\) is consimilar to a product of two coninvolutory affine transformations.
\end{theorem}
Finally,  the following theorem gives that at most four coninvolutions are needed to decompose an arbitrary affine transformation. 
\begin{theorem}\label{Thm3}
Every element \(g = (A,v) \in \mathrm{Aff}(n,\mathbb{C})\) with 
\(|\det(A)|=1\) can be expressed as a product of at most four coninvolutions.
\end{theorem}
{Note that, every coninvolutory matrix has determinant modulus $1$. 
Hence, Theorems~ \ref{Thm2} and~ \ref{Thm3} do not hold when $|\det(A)| \neq 1$.}

 We remark that, while analogous results could in principle be formulated for affine 
transformations over the quaternions, we do not pursue this direction here. 
The main obstruction is that the property of being strongly \(c\)-reversible 
is not invariant under conjugation for quaternionic affine transformations. 
Consequently, a quaternionic analogue would require a suitably conjugation-invariant 
notion of coninvolution, which must first be carefully formulated.

\subsection{Structure of the paper.}
{The paper is organized as follows. In Section \ref{sec:2}, we review the necessary terminology and background. 
Sections \ref{sec:3} and \ref{sec:4} contain our main results on products of two and three coninvolutions (Theorems \ref{Thm1} and \ref{Thm2}). 
Finally in Section \ref{sec:5}, we prove our theorem on the product of four coninvolutions (Theorem \ref{Thm3}).}
    \section{Preliminaries}\label{sec:2}
    \begin{definition}[cf.~{\cite[p.~94]{Ro}}]
Let $\lambda \in \mathbb{C}$. The \emph{Jordan block} $\mathrm{J}(\lambda,m)$ is the $m \times m$ matrix
\[
\mathrm{J}(\lambda,m)=
\begin{pmatrix}
\lambda & 1 & 0 & \cdots & 0 \\
0 & \lambda & 1 & \cdots & 0 \\
\vdots & \ddots & \ddots & \ddots & \vdots \\
0 & \cdots & 0 & \lambda & 1 \\
0 & \cdots & \cdots & 0 & \lambda
\end{pmatrix}.
\]
That is, $\lambda$ appears on the diagonal entries, $1$ on the super-diagonal entries,
and zero elsewhere. A block diagonal matrix whose diagonal blocks are Jordan blocks
is called a \emph{Jordan form}.
\end{definition}

Jordan canonical forms in $\mathrm{GL}(n,\mathbb{C})$ are well studied in the literature;
see {\cite[Chapter~5, Chapter~15]{Ro}}.

Recall that an element $U \in \mathrm{GL}(n,\mathbb{C})$ is called \emph{unipotent}
if all of its eigenvalues are equal to $1$.

\begin{lemma}[cf.~{\cite[Theorem~15.1.1, Theorem~5.5.3]{Ro}}]
Let $A \in \mathrm{GL}(n,\mathbb{C})$ be unipotent. Then there exists
$S \in \mathrm{GL}(n,\mathbb{C})$ such that
\[
SAS^{-1}
=
I_{m_0}
\oplus
J(1,m_1)
\oplus
\cdots
\oplus
J(1,m_k),
\]
where $m_0 \ge 0$, $m_i \ge 2$ for $i \ge 1$, and
\[
m_0 + m_1 + \cdots + m_k = n.
\]
\end{lemma}
We now prove some basic facts concerning conjugate reversibility
and strong conjugate reversibility.
\begin{lemma}
The properties of being $c$-reversible and strongly $c$-reversible
are invariant under conjugation in  $\mathrm{Aff}(n,\mathbb{C})$.
\end{lemma}
\begin{proof}
Suppose $g$ is $c$-reversible. Then there exists $h \in \mathrm{Aff}(n,\mathbb{C})$
such that
\[
h g h^{-1} = \overline{g}^{\, -1}.
\]
If $k \in \mathrm{Aff}(n,\mathbb{C})$ then
$(\overline{k} h k^{-1})(k g k^{-1})(\overline{k} h k^{-1})^{-1}
= \overline{k} h g h^{-1} \overline{k}^{\, -1}
= {k}\,\overline{g}^{\, -1}\,\overline{k}^{\, -1}
= \overline{(k g k^{-1})}^{\, -1}.$
Hence $k g k^{-1}$ is $c$-reversible.

If, in addition, $h \overline{h} = e$, then
$
(\overline{k} h k^{-1}) \, \overline{(\overline{k} h k^{-1})}
= \overline{k} (h \overline{h}) \overline{k}^{-1}
= e.$
Thus strong $c$-reversibility is also preserved under conjugation.
\end{proof}

\begin{lemma}
An element $g \in \mathrm{Aff}(n,\mathbb{C})$ is strongly $c$-reversible
if and only if $g$ is a product of two coninvolution.
\end{lemma}

\begin{proof}
Suppose $g$ is strongly $c$-reversible. Then there exists an element $h \in \mathrm{Aff}(n,\mathbb{C})$ such that 
	\[
	hgh^{-1} = \overline{g}^{-1} \quad \text{with} \quad h\overline{h} = e.
	\]
	 Therefore \(g = (h^{-1})(\overline{g}^{-1}h)\). Since $h$ is a coninvolution, $h^{-1}$ is also a coninvolution. Also, $\overline{g}^{-1}h$ is a coninvolution , because
	\[
	(\overline{g}^{-1}h)\overline{(\overline{g}^{-1}h)} = \overline{g}^{-1}hg^{-1}\overline{h} = h\overline{h} = e.
	\]
	Conversely, suppose $g$ is a product of two coninvolution. That is, suppose $g = ab$, where $a\overline{a} = e$ and $b\overline{b} = e$. Then 
	\[
	a^{-1}ga = ba = \overline{g}^{-1}.
	\]
	Thus, $g$ is strongly $c$-reversible.
\end{proof}

    \section{Product of two coninvolutions}\label{sec:3}
   
We embed $\mathbb{C}^{n}$ into $\mathbb{C}^{n+1}$ as the affine hyperplane
\[
P=\{(z,1)\in \mathbb{C}^{n+1} \mid z\in \mathbb{C}^{n}\}.
\]
Consider the homomorphism 
\[
\Theta : \operatorname{Aff}(n,\mathbb{C}) \longrightarrow \mathrm{GL}(n+1,\mathbb{C})
\]
defined by
\[
\Theta((A,v))=
\begin{pmatrix}
A & v \\
\mathbf{0} & 1
\end{pmatrix},
\]
where $0$ denotes the zero vector in $\mathbb{C}^{n}$. 
The action of $\Theta(\operatorname{Aff}(n,\mathbb{C}))$ on $P$ coincides with the natural action of 
$\operatorname{Aff}(n,\mathbb{C})$ on $\mathbb{C}^{n}$.

In the next two lemmas (\ref{lem:1} and \ref{lem:2}), we classify reversible and strongly reversible elements in the affine group $\operatorname{Aff}(n,\mathbb{C})$.

   \begin{lemma}\label{lem:1}
Let $g=(A,v)\in \mathrm{Aff}(n,\mathbb{C})$. Then $g$ is conjugate reversible in $\mathrm{Aff}(n,\mathbb{C})$ if and only if there exists $h=(B,w)\in \mathrm{Aff}(n,\mathbb{C})$ such that:
\begin{enumerate}[(a)]
    \item $BAB^{-1}=\overline{A}^{-1}$,
    \item $(\overline{A}^{-1}-I_n)w = Bv+\overline{A}^{-1}\overline{v}$.
\end{enumerate}
\end{lemma}
   \begin{proof}
Let $g=(A,v)\in \mathrm{Aff}(n,\mathbb{C})$. Then
\[
\overline{g}^{-1}(z)=\overline{A}^{-1}z-\overline{A}^{-1}\overline{v}.
\]
For $h=(B,w)\in \mathrm{Aff}(n,\mathbb{C})$, we have
\[
hgh^{-1}=\overline{g}^{-1} \;\Longleftrightarrow\; hg=\overline{g}^{-1}h.
\]
Now,
\[
hg(z)=BAz+(Bv+w), \quad 
\overline{g}^{-1}h(z)=\overline{A}^{-1}Bz+\overline{A}^{-1}w-\overline{A}^{-1}\overline{v}.
\]
Comparing terms gives
\[
BAB^{-1}=\overline{A}^{-1}, \quad 
(\overline{{A}^{-1}}-I_n)w = Bv+\overline{A}^{-1}\overline{v}.
\]
This completes the proof.
\end{proof}
    In the next lemma we shall discuss strongly conjugate reversibile of the elements of $\mathrm{Aff}(n,\mathbb{C})$
\begin{lemma}\label{lem:2}
Let $g=(A,v)\in \mathrm{Aff}(n,\mathbb{C})$. Then $g$ is strongly conjugate reversible in $\mathrm{Aff}(n,\mathbb{C})$ if and only if there exists $h=(B,w)\in \mathrm{Aff}(n,\mathbb{C})$ such that:
\begin{enumerate}[(a)]
    \item $BAB^{-1}=\overline{A}^{-1}$ and $B\overline{B}=I_n$,
    \item $(\overline{A}^{-1}-I_n)w = Bv+\overline{A}^{-1}\overline{v}$ and $B\overline{w}+w=0$.
\end{enumerate}
\end{lemma}
 \begin{proof}
   
From Lemma \ref{lem:1}, $g$ is conjugate reversible if and only if
\[
BAB^{-1}=\overline{A}^{-1}
\quad \text{and} \quad
(\overline{A}^{-1}-I_n)w = Bv+\overline{A}^{-1}\overline{v}.
\]
As $g$ is strongly \textit{c}-reversible, therefore
\[
(B,w)(\overline{B},\overline{w}) = (I_n,0).
\]
It implies
\[
B\overline{B}=I_n \quad \text{and} \quad B\overline{w}+w=0.
\]
This completes the proof.
 \end{proof}
 \begin{lemma}[\cite{glm}]
Every element $g \in \mathrm{Aff}(n, \mathbb{C})$ is conjugate to an element of the form 
$g = (A, v)$, where 
\[
A = T \oplus U,
\]
with $T \in \mathrm{GL}(n-m, \mathbb{C})$ and $U \in \mathrm{GL}(m, \mathbb{C})$ such that $T$ has no eigenvalue equal to $1$, and $U$ has only eigenvalue $1$. Moreover, the translation part $v \in \mathbb{C}^n$ can be written as
\[
v = (0, \ldots, 0, v_1, \ldots, v_m),
\]
where $0 \le m \le n$ is the multiplicity of the eigenvalue $1$ of the linear part $A$.

Furthermore, if $1$ is not an eigenvalue of $A$ (i.e., $m = 0$), then $g$ is conjugate to $(A, 0)$.
\end{lemma}
 \par It is difficult to find a strongly $c$-reverser directly when an affine transformation has a non-trivial unipotent part. To overcome this difficulty, we pass to the Lie algebra of the affine group and use the adjoint action.
 We introduce a notion which is similar to adjoint reality of the Lie algebra of classical lie groups (see \cite{gm23}) called adjoint conjugate reality. Let $\mathfrak{g} = \mathfrak{aff}(n, \mathbb{C})$ denote the Lie algebra of the affine group $\mathrm{Aff}(n, \mathbb{C})$, that is,
\[
\mathfrak{g} = \mathfrak{gl}(n, \mathbb{C}) \ltimes \mathbb{C}^n.
\]
The Lie algebra $\mathfrak{aff}(n,\mathbb{C}) = \mathfrak{gl}(n,\mathbb{C}) \ltimes \mathbb{C}^n$ admits a faithful embedding into $\mathfrak{gl}(n+1,\mathbb{C})$ given by
\[
(X,u) \;\longmapsto\;
\begin{pmatrix}
X & u \\
0 & 0
\end{pmatrix}.
\]
Note that the adjoint action of $\mathrm{Aff}(n,\mathbb{C})$ on its Lie algebra $\mathfrak{aff}(n,\mathbb{C})$ is given by
\[
\mathrm{Ad} : G \times \mathfrak{g} \to \mathfrak{g},
\]
\[
\mathrm{Ad}_{(A,v)}(X,u)
=
\left( AXA^{-1},\; A(u - Xv) \right),
\]
where $G = \mathrm{Aff}(n,\mathbb{C})$ and $\mathfrak{g} = \mathfrak{aff}(n,\mathbb{C})$.

An element $X \in \mathfrak{g}$ is called \emph{adjoint $c$-real} (or \emph{adjoint conjugate real}) if there exists $g \in G$ such that
\[
\mathrm{Ad}_g(X) = -\overline{X}.
\]
It is called \emph{strongly adjoint $c$-real} if, in addition, $g \overline{g} = I$. Note that if $\mathrm{Ad}_g(X) = -\overline{X}$, then
\[
g(\exp X)g^{-1} = \exp(-\overline{X}) = (\overline{\exp X})^{-1}.
\] 
In particular, if $X \in \mathfrak{g}$ is adjoint $c$-real (respectively, strongly adjoint $c$-real), then $\exp X$ is $c$-reversible (respectively, strongly $c$-reversible).
\begin{remark}
The converse is not true in general. For example,
\[
A=\operatorname{diag}(r+2\pi i,\,-r), \quad r\in \mathbb{R},
\]
is not \emph{adjoint $c$-real} since
\[
A \not\sim -\overline{A}.
\]
However,
\[
\exp(A)=\operatorname{diag}(e^r,\,e^{-r})
\]
has eigenvalues in reciprocal pairs, and hence
\[
\exp(A) \sim \exp(A)^{-1}=\overline{\exp(A)}^{-1}.
\]
Therefore \(\exp(A)\) is \(c\)-reversible.
\end{remark}
 
 \begin{lemma}\label{lem:3}
     Let $(N,x)\in\mathfrak{aff}(n,\mathbb{C})$ then $(N,x)$ is strongly \emph{adjoint $c$-real} if and only if there exists $(B,w)\in \mathrm{Aff}(n,\mathbb{C})$ which obeys the following conditions.
     \begin{enumerate}[(a)]
         \item $BNB^{-1}=-\overline{N}$ and $B\overline{B}=I_n$,
         \item $Bx+\overline{x}=-\overline{N}w$ and $B\overline{w}+w=0$.
     \end{enumerate}
 \end{lemma}
 
 Lemma \ref{lem:3} can proved in the same line as Lemma \ref{lem:2}. Therefore we omit it here.
 \par Using Lemma \ref{lem:3} we can prove our next Lemma.
 \begin{lemma}\label{lem:4}
     The element $(\mathrm{J}(0,n),x)\in \mathfrak{aff}(n,\mathbb{C})$ is strongly \emph{adjoint $c$-real}.
 \end{lemma}
 \begin{proof}  
{
Let $x=(x_1,\ldots,x_n)^t \in \mathbb{C}^n$ and write $x_n=r_n e^{i\theta}$ with $r_n\in \mathbb{R}^+$. Set $a=e^{-2i\theta}$ and define 
\[
B=\mathrm{diag}(a,-a,a,-a,\ldots).
\]
Define $w=(w_1,\ldots,w_n)^t$ by
\[
w_k =
\begin{cases}
-(a x_k + \overline{x}_k), & k \text{ odd},\\
\;\;a x_k - \overline{x}_k, & k \text{ even}.
\end{cases}
\]
Then a direct computation shows that $BNB^{-1}=-\overline{N}$ and 
$B\overline{w}+w=0$. Moreover,
\[
Bx+\overline{x}=-\overline{N}w.
\]
Hence the conditions of Lemma~\ref{lem:3} are satisfied for $N=\mathrm{J}(0,n)$.
}
\end{proof}

    \begin{lemma}\label{lem:5}
Let $g=(A,v)\in \mathrm{Aff}(n,\mathbb{C})$ with $A=\mathrm{J}(1,n)$. Then $g$ is strongly $c$-reversible in $\mathrm{Aff}(n,\mathbb{C})$.
\end{lemma}
   \begin{proof}
Let $N=J(0,n)$. Then $\exp(N)$ is a unipotent matrix whose Jordan form is
$J(1,n)$. Hence there exists $B\in \mathrm{GL}(n,\mathbb{C})$ such that
\[
B\,\exp(N)\,B^{-1}=A.
\]

Consider the exponential map in $\mathrm{Aff}(n,\mathbb{C})$:
\[
\exp(N,x)=(\exp(N),Cx),
\]
where
\[
C=I+\frac{N}{2!}+\frac{N^2}{3!}+\cdots+\frac{N^{n-1}}{n!}.
\]
Since $C$ is invertible. Hence there exists $x\in\mathbb{C}^n$
such that $Cx=v$.

Now conjugating $\exp(N,x)$ by $(B,0)$ we obtain
\[
(B,0)\exp(N,x)(B,0)^{-1}=(A,v)=g.
\]

Since $(N,x)$ is strongly $c$-real by lemma \ref{lem:4}, it follows that $\exp(N,x)$ is strongly $c$-reversible. and the property is preserved
under conjugation, it follows that $g$ is strongly $c$-reversible.
\end{proof}


  \begin{proposition}\label{pro1}
Let $g=(U,v)\in \operatorname{Aff}(n,\mathbb{C})$, where $U$ is unipotent. 
Then $g$ is strongly $c$-reversible. In particular, $g$ can be written as a 
product of two coninvolutions in $\operatorname{Aff}(n,\mathbb{C})$.
\end{proposition}

\begin{proof}
Since $U$ is unipotent, it is conjugate to a direct sum of Jordan blocks 
$\mathrm{J}(1,n_i)$. By the previous lemma, each affine transformation of the form 
$(\mathrm{J}(1,n_i),v_i)$ is strongly $c$-reversible. Hence $(U,v)$ is strongly 
$c$-reversible. Consequently, $g$ can be written as a product of two 
coninvolutions in $\operatorname{Aff}(n,\mathbb{C})$.
\end{proof}

  \subsection{Proof of \thmref{Thm1}}

We first prove that \((i) \Rightarrow (iii)\).
Let \(g=(A,v)\in \mathrm{Aff}(n,\mathbb{C})\). Since \(g\) is \(c\)-reversible, by lemma ~\ref{lem:1} there exists \(B\in \mathrm{GL}(n,\mathbb{C})\) such that
\[
BAB^{-1}=\overline{A}^{-1}.
\]
This shows that \(L(g)=A\) is \(c\)-reversible in \(\mathrm{GL}(n,\mathbb{C})\).

The implication \((ii)\Rightarrow(i)\) is immediate from the definition of strong \(c\)-reversibility.

Finally, we prove that \((iii)\Rightarrow(ii)\).
  Suppose that the linear part $L(g)=A$ is \textit{c}-reversible. 
Up to conjugacy in $\operatorname{Aff}(n,\mathbb{C})$, we may assume that
\[
g=(A,v)
\]
is of the form
\[
A=
\begin{pmatrix}
T & 0 \\
0 & U
\end{pmatrix},
\qquad
v=
\begin{pmatrix}
0_{\,n-m} \\
\widetilde{v}_{\,m}
\end{pmatrix},
\]
where $T \in \mathrm{GL}(n-m,\mathbb{C})$ has no eigenvalue equal to $1$, 
and $U \in \mathrm{GL}(m,\mathbb{C})$ has $1$ as its only eigenvalue. 
Here $0_{\,n-m}$ denotes the zero vector in $\mathbb{C}^{\,n-m}$ and 
$\widetilde{v}_{\,m} \in \mathbb{C}^{\,m}$.
\\Since $A$ is \textit{c}-reversible, there exists $B \in \mathrm{GL}(n,\mathbb{C})$ such that;
\[
BAB^{-1}=\overline{A}^{-1}.
\]
Let 
\[
B=\begin{pmatrix} B_1 & B_2 \\ B_3 & B_4 \end{pmatrix}.
\]
Then
\[
BA=\begin{pmatrix} B_1T & B_2U \\ B_3T & B_4U \end{pmatrix},\quad
\overline{A}^{-1}B=\begin{pmatrix} \overline{T}^{-1}B_1 & \overline{T}^{-1}B_2 \\ \overline{U}^{-1}B_3 & \overline{U}^{-1}B_4 \end{pmatrix}.
\]
Thus,
\[
B_1T=\overline{T}^{-1}B_1,\;
B_2U=\overline{T}^{-1}B_2,\;
B_3T=\overline{U}^{-1}B_3,\;
B_4U=\overline{U}^{-1}B_4.
\]
Since $\sigma(U)=\{1\}$ and $1 \notin \sigma(T)$, the Sylvester equations
\[
B_2U=\overline{T}^{-1}B_2, \quad B_3T=\overline{U}^{-1}B_3
\]
imply $B_2=0$ and $B_3=0$. Hence
\[
B=\begin{pmatrix} B_1 & 0 \\ 0 & B_4 \end{pmatrix}.
\]
As $B$ is invertible, it follows that $B_1$ and $B_4$ are invertible. Therefore
\[
B_1TB_1^{-1}=\overline{T}^{-1},
\]
so $T$ is $c$-reversible.
Thus $T$ is strongly $c$-reversible in $\mathrm{GL}(n-m,\mathbb{C})$. Hence
\[
T = T_1 T_2,
\]
where $T_1$ and $T_2$ are coninvolutions in $\mathrm{GL}(n-m,\mathbb{C})$.

Now $(U,\tilde{v})$ is also strongly $c$-reversible by Proposition~\ref{pro1}.
Hence
\[
(U,\tilde{v}) = (U_1,v_1)(U_2,v_2),
\]
where $(U_1,v_1)$ and $(U_2,v_2)$ are coninvolutions.

Thus
\[
g=(A,v)=(T_1 \oplus U_1,\, 0 \oplus v_1)\,(T_2 \oplus U_2,\, 0 \oplus v_2)
\]
so $g$ is a product of two coninvolutions. This completes the proof.\qed
\section{Product of three coninvolutions}\label{sec:4}
\begin{lemma}
Suppose $g$ is a product of $(2m+1)$ coninvolutions in $\mathrm{Aff}(n,\mathbb{C})$. Then any consimilar element of $g$ can also be written as a product of $(2m+1)$ coninvolutions.
\end{lemma}

\begin{proof}
Let
\[
g = g_1 g_2 \cdots g_{2m+1},
\]
where each $g_i$ is a coninvolution.

Let $h \in \mathrm{Aff}(n,\mathbb{C})$. Then
\[
h g \overline{h}^{-1}
= (h g_1 \overline{h}^{-1})(\bar{h} g_2 {h}^{-1}) \cdots (h g_{2m+1} \overline{h}^{-1}).
\]
Set
\[
g_i' = h g_i \overline{h}^{-1}, \quad i=1,\dots,2m+1.
\]
Since consimilarity preserves coninvolutions, each $g_i'$ is a coninvolution. Hence
\[
h g \overline{h}^{-1} = g_1' g_2' \cdots g_{2m+1}',
\]
which is a product of $(2m+1)$ coninvolutions.

This completes the proof.
\end{proof}
We shall use the following result of Horn and Johnson in the Lemma \ref{lem:6}.
\begin{lemma}[cf.~{\cite[Lem.~4.6.9]{hj}}]\label{lem:horn}
  Let $A\in \mathrm{M}_n(\mathbb{C)}$ be given. Then $A\bar{A} = I $ if and only if there is a nonsingular
$S \in \mathrm{M}_n(\mathbb{C)}$ such that $A = S \bar{S}^{-1}$.  
\end{lemma}
\begin{lemma}\label{lem:6}
Let \( g \in \mathrm{Aff}(n,\mathbb{C}) \) be coninvolution. Then there exists \( h \in \mathrm{Aff}(n,\mathbb{C}) \) such that
\[
g = h\overline{h}^{-1}.
\]
\end{lemma}

\begin{proof}
If \( g = h\overline{h}^{-1} \), then
\[
g\overline{g} = h\overline{h}^{-1}\,\overline{h}h^{-1} = e.
\]

Conversely, let \( g=(A,v) \in \mathrm{Aff}(n,\mathbb{C}) \) satisfy \( g\overline{g}=e \). Then
\[
A\overline{A}=I \quad \text{and} \quad A\overline{v}+v=0.
\]
By lemma \ref{lem:horn}, there exists \( B \in \mathrm{GL}(n,\mathbb{C}) \) such that
\[
A = B\overline{B}^{-1}.
\]
Define
\[
h=\left(B,\frac{v}{2}\right).
\]
Then
\[
g\overline{h}
= (A,v)\left(\overline{B},\frac{\overline{v}}{2}\right)
= \left(A\overline{B},\, A\frac{\overline{v}}{2}+v\right)
= \left(B,\,-\frac{v}{2}+v\right)
= \left(B,\frac{v}{2}\right)
= h,
\]
where we used \( A\overline{B}=B \) and \( A\overline{v}=-v \). Hence \( g = h\overline{h}^{-1} \).
\end{proof}
\subsection{Proof of \thmref{Thm2}}
Suppose $g = g_1 g_2 g_3$, where each $g_i$ is a coninvolution.

By lemma \ref{lem:6}, there exists $h \in \mathrm{Aff}(n,\mathbb{C})$ such that
\[
g_1 = h \overline{h}^{-1}.
\]
Hence,
\[
h^{-1} g \overline{h} 
= (\overline{h}^{-1} g_2 h)(h^{-1} g_3 \overline{h})
= g_2' g_3'.
\]
Since consimilarity preserves coninvolutions, So $g_2'$ and $g_3'$ are coninvolutions.  
Thus $g$ is consimilar to a product of two coninvolutory affine transformations.

\medskip

Conversely, suppose there exists $h \in \mathrm{Aff}(n,\mathbb{C})$ such that
\[
h g \overline{h}^{-1} = g_1 g_2,
\]
where $g_1, g_2$ are coninvolutions. Then
\[
g = h^{-1} g_1 g_2 \overline{h}
= (h^{-1} g_1 h)\,(h^{-1} g_2 h)\,(h^{-1} \overline{h}).
\]
Set
\[
g_1' = h^{-1} g_1 h, \quad
g_2' = h^{-1} g_2 h, \quad
g_3' = h^{-1} \overline{h}.
\]
Each $g_i'$ is a coninvolution. Hence
\[
g = g_1' g_2' g_3',
\]
so $g$ is a product of three coninvolutions.

This completes the proof. \qed
\section{Product of four coninvolutions}\label{sec:5}
\begin{lemma}
Let $g \in \operatorname{Aff}(n,\mathbb{C})$ be a product of $2m$ coninvolutions. Then any conjugate of $g$ in $\operatorname{Aff}(n,\mathbb{C})$ is also a product of $2m$ coninvolutions.
\end{lemma}
    \begin{proof}
Suppose
\[
g = (a_1 a_2)(a_3 a_4)\cdots(a_{2m-1} a_{2m}),
\]
where each $a_i$ is a coninvolution. Then each product $a_{2i-1}a_{2i}$ is strongly $c$-reversible.

Let $h \in \operatorname{Aff}(n,\mathbb{C})$. Then
\[
hgh^{-1}
= (h a_1 a_2 h^{-1}) \cdots (h a_{2m-1} a_{2m} h^{-1}).
\]
Since strong $c$-reversibility is conjugacy invariant, each
\[
h a_{2i-1} a_{2i} h^{-1}
\]
is again a product of two coninvolutions. Hence $hgh^{-1}$ is a product of $2m$ coninvolutions.
\end{proof}
\begin{remark}
For an odd number of factors, being a product of coninvolutions is not, in general, a conjugacy invariant, whereas for involutions this property is always preserved under conjugation.

\end{remark}
\begin{example}
This example shows that coninvolutions are not conjugacy invariant in 
$\operatorname{Aff}(n,\mathbb{C})$.

Consider $g=(I,v) \in \operatorname{Aff}(n,\mathbb{C})$. Then
\[
g\overline{g}=(I,v)(I,\overline{v})=(I,v+\overline{v}),
\]
so $g$ is a coninvolution if and only if $v+\overline{v}=0$.

Now take $h=(A,0)$ with $A \in \mathrm{GL}(n,\mathbb{C})$. Then
\[
hgh^{-1}=(I,Av).
\]
Moreover,
\[
(Av)+\overline{(Av)} = Av+\overline{A}\,\overline{v}.
\]
Since $v+\overline{v}=0$, we have $\overline{v}=-v$, and hence
\[
(Av)+\overline{(Av)} = (A-\overline{A})v,
\]
which is not zero in general unless $v \in \ker(A-\overline{A})$. 
Thus $hgh^{-1}$ need not be a coninvolution, showing that coninvolutions are not preserved under conjugation.
\end{example}

 \subsection{Proof of \thmref{Thm3}}
In this theorem, we aim to show that every affine transformation 
$g=(A,v) \in \operatorname{Aff}(n,\mathbb{C})$ with $|\det A|=1$ can be written as a product of at most four coninvolutions. 

By the preceding lemma, it suffices to consider the case where, up to conjugacy,
\[
A=
\begin{pmatrix}
T & 0 \\
0 & U
\end{pmatrix},
\qquad
v=
\begin{pmatrix}
0_{\,n-m} \\
\widetilde{v}_{\,m}
\end{pmatrix}.
\]
    Since $U$ is unipotent, $\det U=1$, and hence $|\det T|=1$. {By \cite[Theorem~5]{ba}, $T$ can be written as a product of four coninvolutions}:
\[
T = T_1 T_2 T_3 T_4.
\]
By Proposition~\ref{pro1}, $(U,\widetilde{v})$ is strongly $c$-reversible. Hence there exist $(U_1,v_1),(U_2,v_2)\in \mathrm{Aff}(n,\mathbb{C})$ such that
\[
(U,\widetilde{v}) = (U_1,v_1)(U_2,v_2),
\]
where $(U_1,v_1)$ and $(U_2,v_2)$ are coninvolutions.
\[
g_1 = (T_1 \oplus U_1,\; 0_{\,n-m} \oplus v_1), \quad
g_2 = (T_2 \oplus U_2,\; 0_{\,n-m} \oplus v_2),
\]
\[
g_3 = (T_3 \oplus I_m,\; 0), \quad
g_4 = (T_4 \oplus I_m,\; 0).
\]
Then
\[
g = g_1 \, g_2 \, g_3 \, g_4,
\]
where each $g_i$ is a coninvolution. This completes the proof.\qed

\section*{Acknowledgments}
		Dutta acknowledges the Mizoram University for the Research and Promotion Grant F.No.A.1-1/MZU(Acad)/14/25-26. Gongopadhyay acknowledges ANRF research Grant CRG/2022/003680 and ANRF/ARGM/2025/000122/MTR. Mondal acknowledges the  CSIR grant no. \\ 09/0947(12987)/2021-EMR-I during the course of this work.

	\end{document}